\documentclass[12pt,reqno]{amsart}

\usepackage[cp1251]{inputenc}
\usepackage{amsthm,amssymb,amsmath,amsfonts}
\usepackage{graphicx}
\usepackage[matrix,arrow,curve]{xy}
\newcommand{\CC}{\ensuremath{\mathbb{C}}}

\newcommand{\Q}{\ensuremath{\mathbb{Q}}}

\newcommand{\Z}{\ensuremath{\mathbb{Z}}}

\newcommand{\F}{\ensuremath{\mathbb{F}}}
\newcommand{\Fr}{\ensuremath{\mathbf{F}}}

\newcommand{\SG}{\ensuremath{\mathfrak{S}}}

\newcommand{\ka}{\ensuremath{\Bbbk}}

\newcommand{\kka}{\ensuremath{\overline{\Bbbk}}}

\newcommand{\XX}{{\ensuremath{\overline{X}}}}
\newcommand{\BF}{\ensuremath{\overline{F}}}
\newcommand{\EE}{\ensuremath{\overline{E}}}

\newcommand{\Pro}{\ensuremath{\mathbb{P}}}

\newcommand{\Aut}{\ensuremath{\operatorname{Aut}}}

\newcommand{\Gal}{\ensuremath{\operatorname{Gal}}}

\newcommand{\Hom}{\ensuremath{\operatorname{Hom}}}

\newcommand{\Pic}{\ensuremath{\operatorname{Pic}}}

\newcommand{\Char}{\ensuremath{\operatorname{char}}}

\newcommand{\ord}{\ensuremath{\operatorname{ord}}}

\newcommand{\tp}{{\mathrm T\Pi}}

\makeatletter
\@addtoreset{equation}{section}
\makeatother

\newtheorem{theorem}[equation]{Theorem}
\newtheorem{proposition}[equation]{Proposition}
\newtheorem{lemma}[equation]{Lemma}
\newtheorem{corollary}[equation]{Corollary}

\theoremstyle{definition}
\newtheorem{example}[equation]{Example}
\newtheorem{definition}[equation]{Definition}

\theoremstyle{remark}
\newtheorem{remark}[equation]{Remark}
\newtheorem*{notation}{Notation}

\title{Minimal cubic surfaces over finite fields}

\thanks{The research was carried out at the IITP RAS at the expense of the Russian Foundation for Sciences (project $N^{\underline{o}}$ 14-50-00150).}

\author{Sergey Rybakov and Andrey Trepalin}

\address{\emph{Sergey Rybakov}
\newline
\textnormal{Institute for Information Transmission Problems, 19 Bolshoy Karetnyi side-str., Moscow 127994, Russia}
\newline
\textnormal{\texttt{rybakov.sergey@gmail.com}}}

\address{\emph{Andrey Trepalin}
\newline
\textnormal{Institute for Information Transmission Problems, 19 Bolshoy Karetnyi side-str., Moscow 127994, Russia}
\newline
\textnormal{\texttt{trepalin@mccme.ru}}}

\begin{document}

UDK 512.774.7                                              

\begin{abstract}
Let $X$ be a minimal cubic surface over a finite field $\F_q$. The image $\Gamma$ of the Galois group $\Gal(\overline{\F}_q / \F_q)$ in the group $\Aut(\Pic(\XX))$ is a cyclic subgroup of the Weyl group $W(E_6)$. There are $25$ conjugacy classes of cyclic subgroups in $W(E_6)$, and $5$ of them correspond to minimal cubic surfaces. It is natural to ask which conjugacy classes come from minimal cubic surfaces over a given finite field. In this paper we give a partial answer to this question and present many explicit examples.
\end{abstract}



\maketitle
\section{Introduction}
Let $X$ be a variety over a finite field $\F_q$, and let $N_d$ be the cardinality of the set $X(\F_{q^d})$ of $\F_{q^d}$-points on $X$. The
zeta function of $X$ is the formal power series
$$
Z_X(t)=\exp\left(\sum_{d=1}^\infty \frac{N_dt^d}{d}\right).
$$
In fact, $Z_X(t)$ is always rational (see \cite{SGA4}). If $X$ is a cubic surface, then (see~\cite[IV.5]{Man74})
\begin{equation}\label{zeta}
    Z_X(t)=\frac{1}{(1-t)P(t)(1-q^2t)}
\end{equation}
where
$$
P(t)=\det(1-qt\Fr|\Pic(\XX)\otimes\Q),
$$
\noindent and $\Fr$ is the linear automorphism of $\Pic(\XX)\otimes\Q$ induced by the Frobenius element. 
Manin~\cite{Man63} and Swinnerton-Dyer~\cite{SD} classified possible actions of Frobenius automorphism on $\Pic(\XX)$ preserving the intersection form.
 On the other hand, they did not figure out if a given action comes from an actual cubic 
surface. If the cubic surface $X$ is not minimal, then $X$ is a blow up of a del Pezzo surface $Y$ of higher degree.
In this case~$Y$ is either rational or a minimal del Pezzo surface of degree $4$.
In \cite{Ry05} the first author constructs all types of minimal del Pezzo surfaces of degree $4$ for $q>3$.
In this paper we explicitly construct minimal cubic surfaces with all possible zeta functions over many finite fields. By formula~\eqref{zeta}, the zeta function of a cubic surface is uniquely determined by the polynomial $P(t)$.   There are exactly five such polynomials for minimal cubic surfaces \mbox{(see~\cite[IV.9]{Man74})}: 
\begin{itemize}
\item[$(c_{11})$] $P_{11}(t)=(1-qt)(1+qt+q^2t^2)^3$;
\item[$(c_{12})$] $P_{12}(t)=(1-qt)(1+qt+q^2t^2)(1-qt+q^2t^2)^2$;
\item[$(c_{14})$] $P_{14}(t)=(1-qt)(1+q^3t^3+q^6t^6)$;
\item[$(c_{13})$] $P_{13}(t)=(1-qt)(1+qt+q^2t^2)(1-q^2t^2+q^4t^4)$;
\item[$(c_{10})$] $P_{10}(t)=(1-qt)(1+qt)^2(1+qt+q^2t^2)(1-qt+q^2t^2)$.
\end{itemize}

One can find detailed information about these five cases in Proposition \ref{minclass}.

We say that a cubic surface has type $(c_{i})$ if the polynomial $P_i(t)$ appears in its zeta function.

In his paper \cite{SD10} Swinnerton-Dyer for any finite field $\F_q$ constructs a cubic surface with $q^2 -2q + 1$ points.
 In fact, the only type of cubic surfaces with $q^2 - 2q + 1$ points is~$(c_{11})$. In this paper as a by-product we obtain another construction of such cubic surfaces for odd $q$.
On the other hand, in the paper \cite{BFL16} it is proved that for sufficiently large~$q$ there exists a cubic surface of any given type (see \cite[Theorem 1.7]{BFL16}), and some explicit constructions of non-minimal cubic surfaces are also given.

In this paper we focus on explicit constructions of minimal cubic surfaces. Unfortunately, there remain some restrictions on $q$. The main result of this paper is the following.

\begin{theorem}
\label{MAIN}
For all odd~$q$ there exist cubic surfaces of types $(c_{12})$ and $(c_{13})$.
If \mbox{$q = 6k + 1$}, then there exists a cubic surface of type $(c_{14})$. 
Finally, if $q>2$ there exists a cubic surface of type $(c_{10})$, and there is no such surface over $\F_2$.
\end{theorem}

We prove Theorem \ref{MAIN} in Sections $5$ and $6$ case-by-case.

It is well-known that on a cubic surface divisors $D$ such that $D^2=-2$ and $D \cdot K_X = 0$ form the root system $E_6$ in $\Pic(\XX)\otimes\Q$. This gives a homomorphism from the Galois group $\Gal\left(\kka / \ka \right)$ to the Weyl group~$W(E_6)$.

Iskovskikh and Manin proved that a minimal $\F_q$-rational surface is either isomorphic to a del Pezzo surface or 
admits a structure of a conic bundle. Del Pezzo surfaces of degree greater than $3$ are birationally isomorphic to conic bundles. This observation allowed the first author to construct minimal del Pezzo surfaces of degree $4$ with a given zeta function in~\cite{Ry05}. Minimal cubic surfaces are not birational to conic bundles, and one has to find another way to construct them.
The main idea of this paper is to consider cubic surfaces with Eckardt points. If there is an Eckardt point on a cubic surface then it has a nontrivial automorphism group (see \cite[Proposition~9.1.23]{Dol}). The image of this group in the Weyl group $W(E_6)$ commutes with the image $\Gamma$ of the Galois group $\Gal(\overline{\F}_q / \F_q)$ in $W(E_6)$. So we have some restrictions on the group $\Gamma$ which in turn totally defines the zeta function.

The considered class of cubic surfaces is relatively narrow, and does not give an answer for all possible $q$, but
we hope that other constructions of cubic surfaces could help to fill this gap.

The plan of this paper is as follows.

In Section $2$ we recall some notions and notation about cubic surfaces and present a classification of groups $\Gamma$ such that the corresponding cubic surface is minimal.

In Section $3$ we collect some results about \textit{cyclic cubic surfaces}. These surfaces have non-trivial automorphisms group over $\overline{\F}_q$ and contain at least $9$ Eckardt points. Assuming existence of an elliptic curve with some specific properties, we construct a cyclic cubic surface of type $(c_{14})$.

A cubic surface with an Eckardt point defined over $\F_q$ always has an involution. In Section $4$ we study the twist of the cubic surface by this involution. In particular, we compute the Frobenius action on the Picard group of the twist. We use this observation to construct cubic surfaces of types $(c_{11})$, $(c_{12})$, and $(c_{13})$ starting from specific elliptic curves with some additional data.
Moreover, we show that a cubic surface of type $(c_{10})$ with an Eckardt point exists if and only if there exists a cubic surface with an Eckardt point which is a blowup of $\Pro^2_{\F_q}$ at two points of degree $3$.

In Section $5$ we use the Tate--Honda theory to produce elliptic curves with properties formulated in Sections $3$ and $4$. This completes the constructions of cubic surfaces of types 
$(c_{11})$, $(c_{12})$, $(c_{13})$, and $(c_{14})$.

In Section $6$ for $q > 2$ we find two points of degree $3$ on $\Pro^2_{\F_q}$ such that the blowup of these points is a cubic surface with an Eckardt point. This completes the construction of a cubic surface of type $(c_{10})$. For $q = 2$ we show that such surface does not exist.

For convenience of the reader we give a table of conjugacy classes in $W(E_6)$ and their properties in the appendix.

The first author is a Simons-IUM contest winner and the second author is a Young Russian Mathematics award winner. We would like to thank sponsors and jury of both grants. 
The authors are grateful to Costya Shramov for many useful discussions and comments. 
We thank Alexander Duncan, Sergey Gorchinskiy, and Alexander Kuznetsov for introducing some ideas which were used in this paper.
We are thankful to Michael A. Tsfasman for his interest to this work and for his remarks on the paper.

\begin{notation}

Throughout this paper $X$ is a smooth cubic surface in $\Pro^3_{\F_q}$ over a finite field $\F_q$ of order $q$. The image of the group $\Gal\left( \overline{\F}_q / \F_q\right)$ in the Weyl group $W(E_6)$ acting on the Picard group $\Pic(\XX)$ is $\Gamma$. We denote by $\omega$ a primitive root of unity of order $3$. 

\end{notation}

\section{Cubic surfaces}

In this section we collect some well-known results on cubic surfaces and establish notation.
Let $\ka$ be an arbitrary field, and let $\kka$ be the algebraic closure of $\ka$. 
Let $X$ be a cubic surface in $\Pro^3_{\ka}$, and let $\XX=X\otimes\kka $. Then $\XX$ is isomorphic to the blowup $f: \XX \rightarrow \Pro^2_{\kka}$ of $6$ points $p_1$, $\ldots$, $p_6$ in general position. 
Put $E_i = f^{-1}(p_i)$, and $L = f^*(l)$, where~$l$ is the class of a line on $\Pro^2_{\kka}$ not passing through $p_1$, $\ldots$, $p_6$. The anticanonical class
$$
-K_{\XX} \sim 3L - \sum \limits_{i=1}^6 E_i
$$
is equivalent to a hyperplane section of $\XX$. 
An effective divisor $D$ on $\XX$ of degree $1$ such that $D^2=-1$ is called \emph{a $(-1)$-curve}.
The set of $(-1)$-curves on $\XX$ consists of $E_i$, the proper transforms \mbox{$L_{ij} \sim L - E_i - E_j$} of the lines passing through a pair of points $p_i$ and~$p_j$, and the proper transforms
$$
Q_j \sim 2L + E_j - \sum \limits_{i = 1}^6 E_i
$$
\noindent of the conics passing through five points $p_i$ for $i\neq j$.

In this notation one has:
$$
E_i \cdot E_j = 0; \qquad E_i \cdot L_{ij} = 1; \qquad E_i \cdot L_{jk} = 0;
$$
$$
L_{ij} \cdot L_{ik} = 0; \qquad L_{ij} \cdot L_{kl} = 1; \qquad E_i \cdot Q_i = 0; \qquad E_i \cdot Q_j = 1;
$$
$$
Q_i \cdot Q_j = 0; \qquad Q_i \cdot L_{ij} = 1; \qquad Q_i \cdot L_{jk} = 0,
$$
\noindent where $i$, $j$, $k$, $l$ are pairwise distinct numbers.

The group $\Pic(\XX)$ is generated by the classes $L$ and $E_i$, and the vector space $\Pic(\XX)\otimes\Q$ is of dimension $7$. The set of divisors $D$ such that $D^2=-2$ and $D \cdot K_X = 0$, is the root system $E_6$ in $\Pic(\XX)\otimes\Q$. Denote the image of the corresponding homomorphism from the Galois group $\Gal\left(\kka / \ka \right)$ to the Weyl group~$W(E_6)$ by $\Gamma$.  
Clearly, $\Pic(X) = \Pic(\XX)^{\Gamma}$. The cubic surface $X$ is called \emph{minimal} if $\rho(X) = 1$. In fact, $X$ is not minimal if and only if there exists a $(-1)$-curve $D$ such that for any $\sigma\in\Gamma$ one has either $D = \sigma D$, or $D \cdot \sigma D = 0$.

If the field $\ka$ is finite then the group $\Gamma$ is cyclic. Therefore, a minimal cubic surface over a finite field gives rise to a cyclic subgroup $\Gamma$ in $W(E_6)$ such that $\rho\left(\XX\right)^{\Gamma} = 1$.
Conjugacy classes in $W(E_6)$ were described by Swinnerton-Dyer~\cite{SD}. We give a different proof of classification of cyclic groups $\Gamma$ such that $\rho\left(\XX\right)^{\Gamma} = 1$ to introduce the notation and some properties of these groups.

\begin{remark}\label{S6_in_W6} The order of $W(E_6)$ is equal to $51840 = 2^7 \cdot 3^4 \cdot 5$. There exists an inclusion $\SG_6 \subset W(E_6)$ such that the action on the set of $(-1)$-curves is given as follows: for $\sigma \in \SG_6$ one has $\sigma \left( E_i \right) = E_{\sigma(i)}$, $\sigma \left( L_{ij} \right) = L_{\sigma(i)\sigma(j)}$, and $\sigma \left( Q_i \right) = Q_{\sigma(i)}$. Note that the given subgroup $\SG_6$ is not normal in $W(E_6)$.

\end{remark}

\begin{lemma}
\label{5nonmin}
If a cyclic subgroup $\Gamma\subset W(E_6)$ contains an element of order $5$ \mbox{then $\rho\left(\XX\right)^{\Gamma} > 1$}.
\end{lemma}

\begin{proof}
By the Sylow theorem all subgroups of order $5$ are conjugate in $W(E_6)$. Therefore we can assume that an element of order $5$ is $(12345) \in \SG_6$. This element has two invariant disjoint $(-1)$-curves: $E_6$ and $Q_6$. Therefore this pair of curves is $\Gamma$-invariant, \mbox{and $\rho\left(\XX\right)^{\Gamma} > 1$}.
\end{proof}

\begin{lemma}
\label{2nonmin}
If a cyclic subgroup $\Gamma \subset W(E_6)$ does not contain an element of order $3$ \mbox{then $\rho\left(\XX\right)^{\Gamma} > 1$}.
\end{lemma}

\begin{proof}

If the subgroup $\Gamma$ contains an element of order $5$ then $\rho\left(\XX\right)^{\Gamma} > 1$ by Lemma \ref{5nonmin}.

If the subgroup $\Gamma$ does not contain elements of order $3$ and $5$, then its order is equal to $2^k$ for some $k$. Therefore there exists an $\Gamma$-invariant $(-1)$-curve since the number of $(-1)$-curves is odd. 
Thus $\rho\left(\XX\right)^{\Gamma}> 1$.

\end{proof}

Now we turn our attention to elements of order $3$ in $W(E_6)$.

\begin{lemma}
\label{3class}
There are $3$ conjugacy classes of elements of order $3$ in $W(E_6)$.
\end{lemma}

\begin{proof}
Since, by the Sylow theorem, all subgroups of order $81$ in $W(E_6)$ are conjugate, we have to find any subgroup of order $81$ in $W(E_6)$, and then classify elements of order $3$ in this group up to conjugation in $W(E_6)$.

Since the automorphism group of any given cubic surface maps to $W(E_6)$, to give a description of elements of order $3$ in $W(E_6)$ it is enough to find a subgroup of order $81$ in the group $\Aut(\XX)$ for some cubic $\XX$. We choose the Fermat cubic surface over $\CC$:
$$
x^3 + y^3 + z^3 + t^3 = 0.
$$

We see that the group of order $81$, generated by
$$
(x : y : z : t) \mapsto (\omega x : y : z : t), \qquad (x : y : z : t) \mapsto (x : \omega y : z : t),
$$
$$
(x : y : z : t) \mapsto (x : y : \omega z : t), \qquad (x : y : z : t) \mapsto ( y : z : x : t),
$$
\noindent acts on this surface.

The elements corresponding to $(x : y : z : t) \mapsto ( y : z : x : t)$, and \mbox{$(x : y : z : t) \mapsto ( \omega x : \omega^2 y : z : t)$} are conjugate in $W(E_6)$. Therefore any element of order $3$ in $W(E_6)$ is conjugate to an element corresponding to one of the following automorphisms on the Fermat cubic:
\begin{itemize}
\item[($I$)] $(x : y : z : t) \mapsto (\omega x : y : z : t)$;
\item[($II$)] $(x : y : z : t) \mapsto (\omega x : \omega y : z : t)$;
\item[($III$)] $(x : y : z : t) \mapsto (\omega x : \omega^2 y : z : t)$.
\end{itemize}

\end{proof}

\begin{remark}
\label{3notation}
In what follows we refer to the three conjugacy classes of elements of order~$3$ in $W(E_6)$ as type $I$, type $II$ and type $III$ respectively.
\end{remark}

\begin{remark}
\label{9min}
There is only one cyclic subgroup in $W(E_6)$ of order $9$ up to conjugation, namely the subgroup generated by the element corresponding to the trasformation 
\mbox{$(x : y : z : t) \mapsto ( y : z : \omega x : t)$}. This group contains a subgroup of order $3$ generated by an element of type $I$.
\end{remark}

\begin{lemma}
\label{9max}
A subgroup $\Gamma$ of order $9$ in $W(E_6)$ is not contained in a bigger cyclic subgroup.
\end{lemma}

\begin{proof}
Let $g\in W(E_6)$ be an element such that orbits of $(-1)$-curves for the group generated by $g$ are the following:
$$
E_1 \rightarrow L_{45} \rightarrow L_{35} \rightarrow Q_2 \rightarrow Q_4 \rightarrow L_{16} \rightarrow L_{12} \rightarrow E_5 \rightarrow L_{24} \rightarrow E_1;
$$
$$
E_2 \rightarrow Q_6 \rightarrow L_{36} \rightarrow Q_3 \rightarrow E_4 \rightarrow L_{14} \rightarrow L_{23} \rightarrow L_{46} \rightarrow L_{25} \rightarrow E_2;
$$
$$
E_3 \rightarrow E_6 \rightarrow L_{34} \rightarrow Q_1 \rightarrow L_{56} \rightarrow L_{15} \rightarrow L_{13} \rightarrow Q_5 \rightarrow L_{26} \rightarrow E_3.
$$

By Remark~\ref{9min} we can assume that $g$ is a generator of $\Gamma$. It is straightforward to check that 
$$
E_1 \cdot g^4E_1 = E_2 \cdot gE_2 = E_3 \cdot g^2E_3 = 1;
$$
$$
E_1 \cdot gE_1 = E_1 \cdot g^2E_1 = E_2 \cdot g^2 E_2 = E_2 \cdot g^4 E_2 = E_3 \cdot gE_3 = E_3 \cdot g^4E_3 = 0.
$$
\noindent Therefore these three orbits cannot be permuted by any element $h$ commuting with $g$ since the element $hgh^{-1}$ does not preserve the intersection form. Thus a subgroup \mbox{$\Z / 9\Z \subset W(E_6)$} is not contained in a bigger cyclic subgroup.
\end{proof}

By Lemma \ref{2nonmin}, if $\rho\left(\XX\right)^{\Gamma} = 1$, then $\Gamma$ contains an element of order $3$. This observation is a key to the classification of such subgroups.

\begin{proposition}
\label{minclass}
Let $X$ be a cubic surface over $\F_q$. Then $X$ is minimal if and only if the image~$\Gamma$ of the Galois group $\Gal\left(\overline{\F}_q / \F_q\right)$ in the Weyl group \mbox{$W(E_6)$} satisfy one of the following conditions:
\begin{itemize}
\item[$(c_{11})$] the order of $\Gamma$ is $3$, and $\Gamma$ is generated by an element of type $I$;
\item[$(c_{12})$] the order of $\Gamma$ is $6$, and $\Gamma$ contains an element of type $I$;
\item[$(c_{14})$] the order of $\Gamma$ is $9$, and $\Gamma$ contains an element of type $I$;
\item[$(c_{13})$] the order of $\Gamma$ is $12$, and $\Gamma$ contains an element of type $I$;
\item[$(c_{10})$] the order of $\Gamma$ is $6$, and $\Gamma$ is conjugate to the group generated by $(123)(456)$ and $cs$, where $c=(14)(25)(36)$, and 
	$s$ is an element of $W(E_6)$ such that $sE_i = Q_i$, $sQ_i = E_i$, and $sL_{ij} = L_{ij}$.
\end{itemize}
If two cyclic subgroups $\Gamma_1$ and $\Gamma_2$ satisfy one of these conditions, then $\Gamma_1$ and $\Gamma_2$ are conjugate.
\end{proposition}

This proposition immediately follows from \cite[IV.9. Table 1]{Man74}, where minimal cubic surfaces correspond to classes with index $0$.

First, we prove several lemmas.

\begin{lemma}
\label{3min}
If a cyclic subgroup $\Gamma \subset W(E_6)$ is generated by an element of type $I$, \mbox{then $\rho\left(\XX\right)^{\Gamma} = 1$}.
\end{lemma}

\begin{proof}
As before, we can compute $\rho\left(\XX\right)^{\Gamma}$ for the Fermat cubic $\XX$ over $\CC$, and the geometric action of an element of type $I$. The quotient of $\XX$ by the group $G$ generated by \mbox{$(x : y : z : t) \mapsto (x : y : z : \omega t)$} is nonsingular, since the set of fixed points of $G$ is the curve with equation $t = 0$. Thus the ramification divisor of the morphism $\XX\to \XX/G$ is $-2K_{\XX}$ and, by the Hurwitz formula,
$$
K_{\XX / G}^2 = \frac{\left( 3K_{\XX}\right)^2}{3} = 9.
$$
Thus $\XX / G$ is isomorphic to $\Pro^2_{\CC}$, and $\rho(\XX)^G = \rho(\XX / G) = 1$. 
\end{proof}

\begin{lemma}
\label{3nonmin}
The elements of type $II$ and $III$ are conjugate in $W(E_6)$ to $(123)$ and $(123)(456)$ in $\SG_6 \subset W(E_6)$ respectively.
In particular, if a cyclic subgroup $\Gamma \subset W(E_6)$ is generated by an element of type $II$ or $III$, then $\rho\left(\XX\right)^{\Gamma}> 1$.
\end{lemma}

\begin{proof}
Note that the elements $(123)$ and $(123)(456)$ are not conjugate to an element of type~$I$ by Lemma \ref{3min} since the six disjoint curves $E_i$ are invariant under $(123)$ and $(123)(456)$. Moreover, these two elements are not conjugate to each other since the element $(123)$ has $9$ invariant $(-1)$-curves and the element $(123)(456)$ has no invariant $(-1)$-curves. Therefore these two elements have different types.

Finally, on the Fermat cubic surface the line given by $x = -y$, $z = -t$ is invariant under the action $(x : y : z : t) \mapsto (\omega x : \omega y : z : t)$, and the element $(123)(456)$ has no invariant lines. Therefore an element of type $II$ is conjugate to $(123)$ and an element of type $III$ is conjugate to $(123)(456)$.
\end{proof}

\begin{lemma}
\label{typeIInonmin}
If a cyclic subgroup $\Gamma\subset W(E_6)$ contains an element of type $II$ \mbox{then $\rho\left(\XX\right)^{\Gamma} > 1$}.
\end{lemma}

\begin{proof}

Assume that $\rho\left(\XX\right)^{\Gamma} = 1$. Then, by Lemma \ref{5nonmin} and by Remark \ref{9min}, the order of~$\Gamma$ is equal to $3 \cdot 2^k$. The element of type $II$ is conjugate to $(123)$ and has $9$ invariant \mbox{$(-1)$-curves}. Therefore at least one of these curves is $\Gamma$-invariant since the order of $\Gamma$ is equal to $3 \cdot 2^k$, and an element of order $3$ acts trivially on these curves. This contradicts the assumption that $\rho\left(\XX\right)^{\Gamma} = 1$.

\end{proof}

Recall that $c=(14)(25)(36) \in \SG_6$, and $s\in W(E_6)$ is defined by the relations $sE_i = Q_i$, $sQ_i = E_i$, and $sL_{ij} = L_{ij}$. 
Obviously, $cs=sc$, and the group generated by $cs$ is isomorphic to $\Z/2\Z$.

\begin{lemma}
\label{typeIIImin}
Suppose that a cyclic group $\Gamma \subset W(E_6)$ contains an element of type $III$. Then $\rho\left(\XX\right)^{\Gamma} = 1$ if and only if $\Gamma$ is conjugate to the cyclic group generated by $(123)(456)$ and $cs$.
\end{lemma}

\begin{proof}
By Lemma~\ref{3nonmin}, any element of type $III$ is conjugate to $(123)(456)$. Thus we may assume that $\Gamma$ contains the element $(123)(456)$. 
The group generated by $(123)(456)$ has $6$ orbits consisting of disjoint curves:
\begin{gather}
\label{tp1}
\{E_1, E_2, E_3\}; \qquad \{Q_1, Q_2, Q_3\}; \qquad \{L_{12}, L_{13}, L_{23}\};\\
\label{tp2}
\{Q_4, Q_5, Q_6\}; \qquad \{E_4, E_5, E_6\}; \qquad \{L_{45}, L_{46}, L_{56}\}.
\end{gather}
Take a graph $\tp$ which vertices correspond to these orbits, and two vertices are connected by an edge, if the six curves in the corresponding two orbits are not disjoint. This graph is a triangular prism, and its automorphism group is $\Aut(\tp)\cong\SG_3 \times \Z/2\Z$. Let $G$ be the image of $\Gamma$ in this automorphism group. Note that $(123)(456)$ acts on this graph trivially. 

If $\Gamma$ is conjugate to the cyclic group generated by $(123)(456)$ and $cs$, then $G$ permutes the top~(\ref{tp1}) and the bottom~(\ref{tp2}) of the prism $\tp$. It is clear now that for any line 
$L$ on a cubic there exists $g\in\Gamma$ such that $L\cdot gL=1$. It follows that $X$ is minimal, \mbox{and $\rho\left(\XX\right)^{\Gamma}=1$}.

Assume that $\rho\left(\XX\right)^{\Gamma} = 1$.
 From Lemma \ref{5nonmin} and Remark \ref{9min}, it follows that the order of~$\Gamma$ is equal to $3 \cdot 2^k$; thus $G$ is either trivial or isomorphic to $\Z/2\Z$, since there are no other cyclic subgroups in $\SG_3 \times \Z / 2\Z$ of order~$2^k$. But if $G$ is trivial, then $\rho\left(\XX\right)^{\Gamma} > 1$. Therefore, $G \cong \Z / 2\Z$, and $\ord \Gamma = 6$.

The image $\widetilde{cs}$ of the element $cs$ in $\SG_3 \times \Z/2\Z$ enjoys the following property: for each vertice~$a$ of the graph $\tp$ the vertice $\widetilde{cs}(a)$ is connected with $a$ by an edge, and $\widetilde{cs}$ is the only element of order $2$ in $\SG_3 \times \Z/2\Z$ satisfying this property. 
Any other element $g$ of order $2$ in $\SG_3 \times \Z/2\Z$ either has a vertice fixed by $g$, or has a $g$-invariant pair of vertices, which are not connected by an edge. 
In both cases $\rho\left(\XX\right)^{\Gamma} > 1$, since we can $\Gamma$-invariantly contract either $3$ or $6$ lines.

We claim that there is the unique lift of any element of order $2$ in $\SG_3 \times \Z/2\Z$ to an element of order $2$ in $W(E_6)$. Indeed, let $Z$ be the centralizer of $(123)(456)$. 
Then the kernel of the natural projection from $Z$ to $\Aut(\tp)$ does not contain elements of order~$2$. 
Therefore $\rho\left(\XX\right)^{\Gamma} = 1$ if and only if $\Gamma$ is conjugate to the group generated by $(123)(456)$ and~$cs$.
\end{proof}

\begin{remark}\label{typeIIImin_rem}
It follows that an element of type $III$ cannot commute with an element of order $4$. Indeed, if there exist an element $g$ of order $4$ which commute with an element of type $III$, then $g$ acts on the triangular prism $\tp$ from the proof of the previous lemma. But there are no elements of order $4$ in the group $\Aut(\tp)\cong\SG_3 \times \Z/2\Z$.
\end{remark}

\begin{proof}[Proof of Proposition~\ref{minclass}]  By Lemma~\ref{3min} and Lemma~\ref{typeIIImin} the conditions are sufficient. We prove necessity. 
 By Lemma~\ref{typeIInonmin} and Lemma~\ref{typeIIImin} we may assume that $\Gamma$ contains an element of type $I$.
By Lemma~\ref{3min}, one has $\rho\left(\XX\right)^{\Gamma}= 1$. Now we prove that the order $n$ of~$\Gamma$ is $3$, $6$, $9$, or $12$, and if $\Gamma_1$ and $\Gamma_2$ are cyclic of the same order and contain an element of type $I$, then $\Gamma_1$ and $\Gamma_2$ are conjugate.

The last assertion is clear for $n=3$ and, by Remark \ref{9min}, for $n=9$.
By Lemma \ref{9max}, either $n=9$, or $n=3 \cdot 2^k$.
Assume that $n=3 \cdot 2^k$, and $k>0$.

Let $g\in W(E_6)$ be an element such that orbits of $(-1)$-curves for the group generated by $g$ are the following:
$$
E_6 \rightarrow L_{56} \rightarrow Q_5 \rightarrow E_6;
$$
$$
E_1 \rightarrow Q_2 \rightarrow E_5 \rightarrow L_{45} \rightarrow L_{16} \rightarrow L_{24} \rightarrow L_{15} \rightarrow L_{26} \rightarrow Q_6 \rightarrow E_4 \rightarrow Q_1 \rightarrow L_{13} \rightarrow E_1;
$$
$$
E_2 \rightarrow L_{46} \rightarrow L_{14} \rightarrow E_3 \rightarrow Q_3 \rightarrow L_{12} \rightarrow L_{25} \rightarrow Q_4 \rightarrow L_{23} \rightarrow L_{35} \rightarrow L_{36} \rightarrow L_{34} \rightarrow E_2.
$$

The element $g$ generates a subgroup of order $12$ in $W(E_6)$ containing an element of type $I$. Therefore the order of $\Gamma$ can be $3$, $6$ and $12$.

It is straightforward to check that 
$$
E_1 \cdot gE_1 = E_2 \cdot g^5E_2 = 1;
$$
$$
E_1 \cdot g^5E_1 = E_2 \cdot g E_2 = 0.
$$
\noindent Therefore two orbits of length $12$ cannot be permuted by any element $h$ commuting with $g$ since the element $hgh^{-1}$ does not preserve the intersection form. Thus the order of $\Gamma$ can not be $24$.

Let $g \in \Gamma$ be an element of order $3$ and $h_2 \in \Gamma$ be an element of order $2$. The group $\Gamma$ is cyclic, thus $g$ and $h_2$ commute. There are $27$ lines on $X$, therefore the element $h_2$ has an invariant line $R$. Moreover, the lines $gR$ and $g^2R$ are $h_2$-invariant since $g$ and $h_2$ commute. Let $D$ be any other line. Then $D$ meets exactly one line among $R$, $gR$, and $g^2R$ since $R + gR + g^2R \sim -K_X$. Assume that $D \cdot R = 1$. Then $h_2D \cdot R = h_2D \cdot h_2R = D \cdot R = 1$, and $D + h_2D + R \sim -K_X$. 

Let $S$ be any line which differs from $R$, $gR$, $g^2R$, $D$, $gD$, $g^2D$, $h_2D$, $gh_2D$ and $g^2h_2D$. Then $S$ meets exactly one of the lines $D$, $gD$, $g^2D$ since $D + gD + g^2D \sim -K_X$. We may assume that $S \cdot D = 1$. Note that $D + h_2D + R \sim -K_X$ therefore $S \cdot h_2D = 0$. Thus $h_2S \ne S$ since $h_2S \cdot h_2D = S \cdot D = 1$. So the element $h_2$ nontrivially acts on all lines except $R$, $gR$, $g^2R$, and this action maps any line $D$ meeting $R$ (resp. $gR$, $g^2R$) to $-K_X - D - R$ (resp. $-K_X - D - gR$, $-K_X - D - g^2R$). Therefore up to conjugation there is only one group of order $6$ containing an element of type $I$.

Assume that there is an element $h$ of order $4$ in $\Gamma$. Let $D$ be a line such that $D \cdot R = 1$. Then $hD \cdot R = hD \cdot hR = D \cdot R = 1$. There are ten lines meeting $R$ on $X$. Thus we have six possibilities of $hD$ since the other four lines are $gR$, $g^2R$, $D$ and $h^2D$. But these six possibilities give one conjugacy class of cyclic subgroups in $W(E_6)$, since we can change $h$ by $h^3$, and take a conjugation of $h$ by an element of a group of order $9$ containing $g$. Therefore up to conjugation there is only one group of order $12$ containing an element of type $I$.

\end{proof}

\section{Cyclic cubic surfaces}

This section is devoted to a special class of cubic surfaces known as \textit{cyclic cubic surfaces}. Let $\ka$ be a field, and let $X$ be a smooth cubic surface in $\Pro^3_{\ka}$. As before, we denote by~$\Gamma$ the image of the Galois group $\Gal\left(\overline{\ka} / \ka\right)$ in the Weyl group $W(E_6)$. We assume that $\Char \ka \ne 2$, and $\Char \ka \ne 3$.

\begin{definition}
\label{cycliccubic} Let $f\in\ka[x,y,z]$ be a homogeneous polynomial of degree $3$.
The surface $X$ in $\Pro^3_{\ka}$ given by the equation
$$
f(x, y, z) + t^3 = 0
$$
\noindent is called \textit{a cyclic cubic surface}.
\end{definition}

The projection $\pi:(x : y : z : t) \mapsto (x : y : z)$ induces a finite morphism \mbox{$X\to\Pro^2_{\ka}$} of degree $3$ branched in the elliptic curve $E$ given by the equation $f(x, y, z) = 0$. 
Denote by $F\subset X$ the elliptic curve given by the equation $t=0$. Clearly, $\pi$ induces an isomorphism~$F\cong E$.

\begin{lemma}
If $\omega\in\ka$, then the transformation $g: (x : y : z : t) \mapsto (x : y : z : \omega t)$ induces an element of type $I$ in $W(E_6)$
(see  Remark \ref{3notation}).
\end{lemma}

\begin{proof}

The transformation $g$ generates a group $G$ of order $3$. Clearly, $X / G \cong \Pro^2_{\ka}$. Therefore $\rho(X)^G = \rho(\Pro^2_{\ka}) = 1$. By Lemma \ref{3nonmin}, the type of $g$ is not $II$ or $III$. Thus the type of $g$ is $I$.

\end{proof}

\begin{definition}
A point on a cubic surface $X$ is called \textit{an Eckardt point} if there exist three lines on $X$ passing through this point.
\end{definition}

\begin{lemma}
\label{9Eckardt}
Exactly $9$ Eckardt points of $X$ lie on $F$.
\end{lemma}

\begin{proof}
Let $\XX\cong X\otimes_\ka\kka$ be a cyclic cubic surface over $\kka$. Denote by $\BF\subset\XX$ the elliptic curve corresponding to $F$.
Assume that $P \in \BF$ is an Eckardt point on $\XX$, and let $T_P$ be a tangent plane at the point $P$. Then $T_P$ contains three lines on $\XX$ and $\pi(T_P)$ is a line on~$\Pro^2_{\kka}$. This line intersects the curve $\EE = E\otimes_{\ka}\kka$ at a point $\pi(P)$ with multiplicity $3$. Thus $\pi(P)$ is an inflection point of $\EE$. Conversely, if $\pi(P)$ is an inflection point of $\EE$, then $P$ is an Eckardt point 
on $\XX$. Any plane elliptic curve has exactly $9$ inflection points. The lemma follows.
\end{proof}

Note that the line passing through any two inflection points on $E$ meets this curve in a third point, which is an inflection point.

We use the following lemma in Section 5.4 to construct a minimal cubic surface of type~$(c_{14})$.

\begin{lemma}
\label{inf9}
Assume that the action of $\Gal\left( \F_{q^3} / \F_q \right)$ on the set of inflection points on $E$ has three orbits of length $3$. If there exists an orbit consisting of three inflection points  not lying on a line, then $X$ is minimal, and its type is $(c_{14})$.
\end{lemma}

\begin{proof}

Assume that $\Gamma$ does not contain elements of order $9$. By assumption, there exists $g\in\Gamma$ of order $3$. 
Denote by $A$, $B$ and $C$ the three $(-1)$-curves passing through an Eckardt point $P$. One has
$$
A + B + C \sim gA + gB + gC \sim g^2A + g^2B + g^2C \sim -K_\XX.
$$
\noindent 
Since $A\neq gA$, exactly one of the intersection numbers  $gA\cdot A$, $gB\cdot A$, and $gC\cdot A$ is equal to $1$, and other are zero.
The same is true for the lines $g^2A$, $g^2B$, and $g^2C$. 
If $A \cdot gA = 1$, then $gA \cdot g^2 A = g^2A \cdot A =1$; thus $A + gA + g^2A \sim -K_X$ is a hyperplane section. 
Thus the points $P$, $gP$, and $g^2P$ lie on a line, which is the intersection of this hyperplane and of the hyperplane given by the equation $t=0$. 
This contradiction proves that $A \cdot gA = 0$, and, similarly, $B \cdot gB= 0$ and $C \cdot gC = 0$.

Assume that $A \cdot gB = 1$. In this case, $C \cdot gB = 0$, and, since $C \cdot gC = 0$, we have $C \cdot gA = 1$. Likewise $B \cdot gC = 1$.
Thus $g^2C \cdot gB = g^2C \cdot A = 1$; in other words, $A + gB + g^2C \sim -K_X$ is a hyperplane section. Therefore the points $P$, $gP$ and $g^2P$ lie on a line. A contradiction.

Thus $\Gamma$ contains an element of order $9$, and, by Lemma \ref{9max}, it is generated by this element, and, by Remark \ref{9min}, contains an element of type~$I$. 
By Lemma~\ref{3min}, $X$ is minimal.

\end{proof}

\section{Cubic surfaces with an Eckardt point}
\subsection{}
A smooth cubic surface $X$ over $\F_q$ with an Eckardt point defined over $\F_q$ has an involution.
In this section we compute the action of the Galois group on the twist of $X$ by this involution.
Later, in Sections $5$ and $6$, we use these results to construct cubic surfaces satisfying conditions of the cases $(c_{11})$, $(c_{12})$, $(c_{13})$ and $(c_{10})$.

\begin{proposition}[{cf. \cite[Proposition~9.1.23]{Dol}}]
\label{EckInv}
Let $X$ be a cubic surface over arbitrary field $\ka$ and $g$ be an involution of $X$ which fixes a hyperplane section and an isolated point. Then the isolated fixed point of $g$ is an Eckardt point of $X$ defined over $\ka$, and such involutions bijectively correspond to Eckardt points of $X$ defined over $\ka$. 
\end{proposition}

\begin{lemma}
\label{Eckinv}
The involution $g$ has type $c_3$ (see Table \ref{table1}).
\end{lemma}

\begin{proof}
By definition, the fixed locus of $g$ consists of a hyperplane and an isolated point. By the Hurwitz formula one has
$$
K_{\XX/\langle g \rangle}^2 = \frac{1}{2} \left( 2K_{\XX} \right)^2 = 6,
$$
\noindent and the quotient $\XX / \langle g \rangle$ is a singular del Pezzo surface of degree $6$ with one $A_1$ singular point. 
Thus $\rho(\XX / \langle g \rangle) = 3$.

From Table \ref{table1} one can see that $c_3$ is the only conjugacy class of elements of order $2$ such that the invariant Picard number is equal to $3$.
\end{proof}

\begin{lemma}
\label{UniInv}
Let $X$ be a cubic surface over $\F_q$, and let $P \in X(\F_q)$ be an Eckardt point. Denote by $R_1$, $R_2$ and $R_3$ three lines on $\XX$ passing through $P$. Assume that $h_1$ and $h_2$ are (geometric or arithmetic) automorphisms of $\XX$ of type $c_3$ such that $h_1R_i = h_2R_i=R_i$. Then the images of $h_1$ and $h_2$ are equal in $W(E_6)$.
Moreover, any hyperplane section containing $R_i$ and two other lines on $\XX$ is $h_1$-invariant.
\end{lemma}

\begin{proof}

One can consider the Fermat cubic surface over $\CC$ with its automorphism of type~$c_3$ given by $(x:y:z:t) \mapsto (y:x:z:t)$ and check that this automorphism has exactly three invariant lines and each other line $D$ maps to a line meeting $D$ at a point.

Let $D$ be any line on $\XX$ that differs from $R_i$. Then $h_1 D \ne D$, and $D$ meets exactly one of the lines $R_i$, since $R_1 + R_2 + R_3 \sim -K_X$.
Let $R_j \cdot D = 1$. Then $D \cdot h_1D = 1$, and $R_j \cdot h_1D = h_1R_j \cdot h_1D = 1$. Therefore for each $D$ meeting $R_j$ at a point there is only one possibility for the image $h_1D \sim -K_X - R_j - D$.
This argument shows that $h_1D=h_2D$ for all $D$. This proves the lemma.
\end{proof}

\subsection{Forms of algebraic varieties}

Let $X_1$ and $X_2$ be two algebraic varieties over an algebraically non-closed field $\ka$, and let $L$ be a finite Galois extension of $\ka$. If $X_1 \otimes L$ and $X_2 \otimes L$ are isomorphic then we say that $X_2$ is \textit{an $L$-form} of $X_1$.

The set of $L$-forms of a variety $X$ is isomorphic to $H^1\left(\Gal(L/\ka), \Aut(X \otimes L) \right)$. For any $\varphi\in H^1\left(\Gal(L/\ka), \Aut(X \otimes L) \right)$ the corresponding variety is called \emph{the twist of $X$ by~$\varphi$}.

\begin{proposition}
\label{dPtwist}
Let $X_1$ be a smooth algebraic variety over $\F_q$ such that a cyclic group~$G$ acts on $X_1$ and this action induces a faithful action of $G$ on the group $\Pic(\XX_1)$. Let $\Gamma_1$ be the image of the Galois group $\Gal\left(\overline{\F}_q / \F_q \right)$ in the group $\Aut\left(\Pic(\XX_1)\right)$. Let $\gamma$ and $g$ be the generators of $\Gamma_1$ and $G$ respectively.

Then there exists a variety $X_2$ such that the image $\Gamma_2$ of the Galois group $\Gal\left(\overline{\F}_q / \F_q \right)$ in the group $\Aut\left(\Pic(\XX_2)\right) \cong \Aut\left(\Pic(\XX_1)\right)$ is generated by the element $g\gamma$.

\end{proposition}

\begin{proof}
Let $m$ be the order of $G$. Choose $n$ such that both $m$ and $\ord \Gamma$ divide $n$.
The inclusion $G\to \Aut\left(\Pic(\XX_1)\right)$ induces the map $$H^1(\Gal(\F_{q^n}/\F_q),G)\to H^1(\Gal(\F_{q^n}/\F_q),\Aut\left(\Pic(\XX_1)\right)).$$
Since the action of $\Gal(\F_{q^n}/\F_q)$ on $G$ is trivial, the group $H^1(\Gal(\F_{q^n}/\F_q),G)$ is isomorphic to $\Hom(\Gal(\F_{q^n}/\F_q),G)$. 
On the other hand, the orders of both groups are equal to $m$, therefore the group $\Hom(\Gal(\F_{q^n}/\F_q),G)$ is also cyclic of order $m$ and generated by the map $\varphi$  given by the formula:
$\varphi(\sigma^r)=g^r$, where $\sigma$ generates $\Gal(\F_{q^n}/\F_q)$. Let $X_2$ be the twist of $X_1$ by the cocycle $\varphi$. Then the Galois module $\Pic(\XX_2)$ is a twist of $\Pic(\XX_1)$ by~$\varphi$. 
It follows that $\Gamma_2$ is generated by the element $g\gamma$.
\end{proof}

We need the following corollary of Proposition \ref{dPtwist} to construct minimal cubic surfaces of types $(c_{11})$, $(c_{12})$ and $(c_{10})$.

\begin{proposition}
\label{twist-to-use}
~

\begin{enumerate}
\item There exists a cubic surface of type $(c_{11})$ with an Eckardt point defined over $\F_q$, if and only if there exists a cubic surface of type $(c_{12})$ with an Eckardt point defined over $\F_q$.

\item There exists a cubic surface of type $(c_{10})$ with an Eckardt point defined over $\F_q$, if and only if there exists a cubic surface $X$ with an Eckardt point defined over $\F_q$ such that the image $\Gamma$ of the Galois group $\Gal\left(\overline{\F}_q / \F_q \right)$ in the group $\Aut\left(\Pic(\XX)\right)$ is generated by an element of type $III$.
\end{enumerate}

\end{proposition}

\begin{proof}
If we have a cubic surface $X_1$ with an Eckardt point defined over $\F_q$, then, by Lemma~\ref{Eckinv}, the group $G$ of order $2$ generated by an element $g$ of type $c_3$ acts on $X_1$. Let~$\Gamma_1$ be the image of the Galois group $\Gal\left(\overline{\F}_q / \F_q \right)$ in the group $\Aut\left(\Pic(\XX_1)\right)$, and~$\gamma$ be the generator of $\Gamma_1$. By Proposition \ref{dPtwist} there exists a cubic surface $X_2$ such that the image $\Gamma_2$ of the Galois group $\Gal\left(\overline{\F}_q / \F_q \right)$ in the group $W(E_6)$ is generated by the element~$g\gamma$.

Since the elements $\gamma$ and $g$ commute in $W(E_6) \subset \mathrm{GL}_6(\mathbb{Z})$, these elements can be simultaneously diagonalized in $\mathrm{GL}_6(\mathbb{C})$. Let $\lambda_1$, $\dots$, $\lambda_6$ be the eigenvalues of $\gamma$, and let $\mu_1$, $\dots$, $\mu_6$ be the eigenvalues of $g$. Then there exists a permutation $\sigma \in \SG_6$ such that the eigenvalues of $g\gamma$ are $\lambda_1\mu_{\sigma(1)}$, $\dots$, $\lambda_6\mu_{\sigma(6)}$.

If $\ord \Gamma_1 = 3$ and $\gamma$ has type $I$ or $III$, then the proposition follows from the careful analysis of Table \ref{table1}. If $\gamma$ has type $c_{10}$ or $c_{12}$, then, by Lemma \ref{UniInv}, actions of elements~$\gamma^3$ and $g$ on $\Pic(\XX)$ coinside in $W(E_6)$. Thus $g\gamma$ is an element of order $3$ and of type~$I$ or $III$ respectively.
\end{proof}

\subsection{}\label{cubic_construction}
The quotient of $X$ by the involution $g$ is a del Pezzo surface of degree $6$ with a singular point of type $A_1$. The blowup of this singular point is a smooth surface $Y$ which is the blowup of $\Pro^2_{\F_q}$ at three points, say $T_1, T_2$ and $T_3$, lying on a line $W$. Denote by~$\pi$ the (non-regular) transformation from $X$ to $\Pro^2_{\F_q}$. The irregularity locus is the Eckardt point $P$.
Let $E\subset \Pro^2_{\F_q}$ be the image of the fixed hyperplane section of $g$ under $\pi$. It is not hard to see that $E$ is a smooth curve of genus $1$, and that $E\cap W=\{T_1,T_2,T_3\}$. 
For any line $D$ on $X$ such that $P\not\in D$ its image $\ell=\pi(D)$ is a line on $\Pro^2_{\F_q}$ passing through $T_i$ for some $i$, and tangent to $E$ at some other point $P_i$. 
We say that such a line $\ell$ on $\Pro^2_{\F_q}$ is \emph{distinguished}. Denote by $T(\ell)$ the corresponding point $T_i$ and by $P(\ell)$ the point $P_i$. 
There are exactly four distinguished lines for each of the points $T_1$, $T_2$, and $T_3$.

We can inverse this construction. Start with an elliptic curve $E$ in $\Pro^2_{\F_q}$ and a line~$W$, which intersects $E$ at three points $T_1$, $T_2$ and $T_3$. Then blow up these three points, contract the transform of $W$, and obtain a singular del Pezzo surface $Y$ of degree $6$. Take a double cover $X$ of $Y$, branched in the transform of $E$. The surface $X$ is a smooth cubic surface with an Eckardt point defined over $\F_q$.

\begin{lemma}\label{triangle_lemma}
Let $D_1, D_2$ and $D_3$ be three lines on $X$ such that $D_1+D_2+D_3$ is a hyperplane section, and $P\not\in D_1\cup D_2\cup D_3$. Then (possibly after renumeration) $T_i\in\pi(D_i)$. 
Moreover, points $P(\pi(D_1)), P(\pi(D_2))$, and $P(\pi(D_3))$ lie on a line.

Conversely, suppose that points $P_1, P_2, P_3\in E$ lie on a line. Take a distinguished line  $\ell_i$ such that $P(\ell_i)=P_i$, and $T(\ell_i)=T_i$ for all $i\in\{1,2,3\}$. 
Then the divisor $\pi^{-1}_*(\ell_1)+\pi^{-1}_*(\ell_2)+\pi^{-1}_*(\ell_3)$ is a sum of two hyperplane sections.
\end{lemma}

\begin{proof}
Let $H$ be the $g$-invariant hyperplane section in $\Pro^3_{\F_q}$. Then the map $\pi$ is isomorphic to the projection from the point $P$ to the plane $H$. 
Using this isomorphism we identify the points $T_1$, $T_2$ and $T_3$ with intersection points of $H$ and the lines $R_1$, $R_2$ and $R_3$ passing through $P$. 
Similarly, $W$ can be identified with the intersection of $H$ and the tangent plane $T_{P} X$, and $E$ with the intersection of $H$ and $X$.

Now it is clear, that a line $D_i$ such that $D_i \cdot R_i = 1$ meets the curve $E$ at a $g$-fixed point $D_i \cap gD_i$. Therefore $\pi(D_i)$ passes through $T_i$ and $P_i = \pi\left(D_i \cap gD_i\right)$.  The points $P_1$, $P_2$ and $P_3$ lie on a line, which is the intersection of $H$ and the plane containing $D_1$, $D_2$ and~$D_3$.

Take two hyperplane sections $H_i$ through $P$ and $\ell_i$ for $i=1$ and $i=2$. The intersection of $X$ and $H_i$ is a triple of lines $R_i$, $D_i$, and $D'_i$. The line $D_1$ intersects 
$D_2$ or $D'_2$. We may assume that $D_1 \cdot D_2=1$. Thus there is the third line $D_3$ in the hyperplane section through $D_1$ and $D_2$, and, similarly, the line $D'_3$ in the hyperplane section through $D'_1$ and~$D'_2$.
\end{proof}

This lemma motivates the following definition.
We say that three distinguished lines $\ell_1$, $\ell_2$ and $\ell_3$ form a \emph{triangle}, if $T(\ell_i)=T_i$, and the points $P(\ell_1)$, $P(\ell_2)$ and $P(\ell_3)$ lie on a line.  

\begin{proposition}
\label{cases_c11_c13}
Suppose there exists an elliptic curve $E\subset \Pro^2_{\F_q}$ and a line $W\subset \Pro^2_{\F_q}$ such that the three points $T_1, T_2$, and $T_3$ of the intersection of $E$ and $W$ form a Galois orbit.

If all distinguished lines are not defined over $\F_{q^3}$, then the corresponding cubic surface has type $(c_{13})$.

If all distinguished lines are defined over $\F_{q^3}$, and the set of distinguished lines is the union of $4$ Galois orbits, that are triangles of distinguished lines, 
then the corresponding cubic surface has type $(c_{11})$ or $(c_{12})$.
\end{proposition}

\begin{proof}

Let $X$ be the cubic surface constructed in Subsection~\ref{cubic_construction}. The lines $R_1$, $R_2$ and $R_3$ passing through $P$ are defined over $\F_{q^3}$ and not defined over $\F_q$. Therefore any element of order $3$ in $\Gamma$ has type $I$ or $III$ since an element of type $II$ does not have an orbit consisting of three meeting each other lines. Therefore the group $\Gamma$ of a cubic surface $X$ can have only types $c_9$, $c_{10}$, $c_{11}$, $c_{12}$, $c_{13}$, $c_{22}$ or $c_{23}$ (see Table \ref{table1}).

Note that if $\Gamma$ contains an element of even order whose type is not $c_3$ then this element nontrivially acts on the set of distinguished lines. 
Therefore not all these lines are defined over $\F_{q^3}$ for cases $c_{13}$, $c_{22}$ and $c_{23}$. 
Moreover, if $X$ has type $c_{13}$ then all distinguished lines are not defined over $\F_{q^3}$ and are defined over $\F_{q^{6}}$ by Lemma~\ref{UniInv}.

For cases $c_{22}$ and $c_{23}$ there exists a distinguished line defined over $\F_{q^3}$, since in case $c_{22}$ the element of order $2$ in $\Gamma$ has type $c_{16}$ (see Table \ref{table1})
 and has $9$ invariant lines on $X$. In case~$c_{23}$ the element of order $2$ in $\Gamma$ has type $c_{17}$, which is a product of elements of types~$c_3$ and $c_{16}$, 
and we can reduce this case to the previous by applying Proposition \ref{dPtwist}.

Assume that all distinguished lines are defined over $\F_{q^3}$. 
If $X$ has type $c_9$ or $c_{10}$, then~$\Gamma$ contains an element of type $III$; thus there is a triple of lines $D_1$, $D_2$ and $D_3$ on $X$, 
which is defined over $\F_{q^2}$, and $D_i \cdot D_j=0$, if $i\neq j$. In particular, they does not lie in a hyperplane section. 
Therefore for a such triple the corresponding points $P_1$, $P_2$ and $P_3$ on~$E$ do not lie on a line. 
It follows that $X$ has type $(c_{11})$ or $(c_{12})$.

\end{proof}

\section{Elliptic curves and cubic surfaces}\label{ex9}
\subsection{}
In this section we assume that $q$ is odd.
A curve $E$ of genus $1$ over a finite field has a rational point $O\in E(\F_q)$. 
Thus it is an elliptic curve, and there exists a group law $+_O$ on $E$ such that $O$ is the zero with respect to~$+_O$. 

Suppose that we are given an embedding of $E$ into a projective plane.
Let \mbox{$D=\sum_{P\in E(\overline{\F}_q)} a_P P$} be a hyperplane section of $E$, and let $Q=\sum_{O,P\in E(\overline{\F}_q)} a_P P\in E(\F_q)$ be the sum with respect to the group law $+_O$. 
Three points $P_1,P_2$, and $P_3$ lie on one line if and only if $P_1+_OP_2+_OP_3=Q$. In particular, a point $P\in E(\overline{\F}_q)$ is an inflection point if and only if $3P=Q$.

Fix a prime $\ell$ which is not equal to $\Char(\F_q)$. We denote by $E[\ell^n]$ the $\ell^n$-torsion group subscheme of $E$. Then $E[\ell^n](\F_q)\subset E(\overline{\F}_q)$ is the subgroup of points 
annihilated by $\ell^n$. The multiplication by $\ell$ induce a sequence of homomorphisms \mbox{$E[\ell^n](\F_q)\to E[\ell^{n-1}](\F_q)$}. 
The module $T_\ell E=\varprojlim_n E[\ell^n](\F_q)$ is called \emph{the Tate module of $E$} and is endowed with a natural 
linear action of $\Gal\left(\overline{\F}_q/\F_q\right)$. This is a free module of rank~$2$ over the ring of $\ell$-adic numbers~$\Z_\ell$.
 The Frobenius element $\sigma\in\Gal\left(\overline{\F}_q/\F_q\right)$ induce a semisimple linear endomorphism $\Fr:T_\ell E\to T_\ell E$. The characteristic polynomial $f_E(t)=\det(t-\Fr)$ is called \emph{the Weil polynomial of $E$}. The Weil polynomial of an elliptic curve is equal to $f(t)=t^2 - bt + q$ for some $b\in\Z$ such that $|b|\le 2\sqrt{q}$.
The Deuring theorem~\cite{Deu41} (which was reproved by Waterhouse using Honda-Tate theory~\cite{Wa}) gives a classification of possible Weil polynomials of elliptic curves. In what follows we need the following corollary of this theorem.

\begin{theorem}[Deuring]\label{Deuring}
Suppose $f(t)=t^2 - bt + q$, where $b\in\Z$ and $|b|\le 2\sqrt{q}$. If $(q,b)=1$, or if $\Char\F_q=3$ and $b=\pm\sqrt{3q}$, then there exists an elliptic curve  $E$ over $\F_q$ with the Weil polynomial $f$. 

\end{theorem}

Let $E$ be an elliptic curve with the Weil polynomial $f$. Then the Weil polynomial of $E\otimes \F_{q^3}$ is equal to $f_3(t)=t^2-(b^3-3qb)t+q^3$.
The order of the group $E(\F_q)$ of $\F_q$-points is equal to $f_E(1)=q-b+1$. 

\subsection{Case $(c_{13})$}
\begin{proposition}\label{case13}
Suppose that there exists an elliptic curve $E$ such that $E[4](\F_{q^2})=0$, and $E[4](\F_{q^3})=(\Z/2\Z)^2$, and a point $P\in E[4](\F_{q^6})$ of order $4$. Then there exists a cubic surface of type $(c_{13})$.
\end{proposition}
\begin{proof}
Let $Q_1,Q_2,Q_3\in E[2](\F_{q^3})$ be non-zero points. Clearly, $Q_1+_OQ_2+_OQ_3=O$.
The degree of $P$ is $6$, because $E[4](\F_{q^2})=0$, and $E[4](\F_{q^3})=(\Z/2\Z)^2$.
We embed $E$ to $\Pro^2$ by the divisor $D=3O$, and take a line through $Q_1$ tangent to $E$ at $P$.
We can apply Proposition~\ref{cases_c11_c13}. 
\end{proof}

\begin{example}
By the Deuring theorem, the polynomial $f(t)=t^2-t+q$ is a Weil polynomial of an elliptic curve such that $E[2](\F_q)=0$, and $E[2](\F_{q^3})=(\Z/2\Z)^2$ 
since $f(1)=q$, and $f_3(1)=q(q^2+3)\equiv 4\bmod 8$. Moreover, the number of points on $E\otimes \F_{q^2}$ is equal to $q(q+2)$; thus $E[2](\F_{q^2})=0$. 
Finally, there are at least $8$ points on $E\otimes \F_{q^6}$; therefore there exists a point $Q\in E[4](\F_{q^6})$ of order $4$. By Proposition~\ref{case13}, there exists a cubic surface of type $(c_{13})$.
\end{example}

\subsection{Cases $(c_{11})$ and $(c_{12})$}
\begin{proposition}\label{case11}
Suppose that there exist a prime $\ell>2$ and an elliptic curve $E$ such that $E[2\ell](\F_q)=0$, and $E[2](\F_{q^3})=(\Z/2\Z)^2$, and a point $Q\in E[\ell](\F_{q^3})$. Then there exist cubic surfaces of types $(c_{11})$ and $(c_{12})$.
\end{proposition}
\begin{proof}
The point $Q$ generates a cyclic group of order $\ell$, and in this group there exists a point $P_1$ such that $2P_1+Q=0$. 
Embed the curve $E$ by the divisor $D=3O$. We know that there exists a line $L_1$ passing through $Q$ and tangent to $E$ at $P_1$.
Put $Q_1=Q$, $Q_2=F(Q)$, $Q_3=F^2(Q)$, $P_2=F(P_1)$, and $P_3=F(P_2)$. The sets $\{Q_1,Q_2,Q_3\}$ and $\{P_1,P_2,P_3\}$ are Galois orbits.
By assumption, $F-1$ acts bijectively on $E[\ell](\F_{q^3})$. Since $(F^3-1)(P_1)=0$, we also have $P_1+F(P_1)+F^2(P_1)=0$ in $E[\ell](\F_{q^3})$; this proves that the points $P_1$, $P_2$, and $P_3$ lie on a line. 

Let $L'_1$ be a line passing through $Q_1$ and tangent to $E$ at a point $P'_1$. Then \mbox{$S=P'_1-P_1\in E[2](\F_{q^3})$}. Denote by $\{P'_1,P'_2,P'_3\}$ the corresponding Galois orbit. Since $E[2](\F_q)=0$, we have $S+F(S)+F^2(S)=0$, and thus the points $P'_1$, $P'_2$, and $P'_3$ lie on a line as well. Now use Proposition~\ref{cases_c11_c13}. 
\end{proof}

\begin{example}
If $f(t)=t^2-bt+q$ is a Weil polynomial such that $b$ is odd, then \mbox{$E[2](\F_q)=0$}, and $E[2](\F_{q^3})=(\Z/2\Z)^2$. Let $b=1$.
Note that $q^2+3\equiv 4\bmod 8$. In particular, $q^2+3$ is divisible by an odd prime. If $q>3$, then there exists $\ell>2$ which does not divide $f(1)=q$, but divides $f_3(1)=q(q^2+3)$. 
Finally, if $q=3$, and $b=3$, then $f(1)=1$, and $f_3(1)=28$; in this case, we take $\ell=7$ . 
By the Deuring theorem there exists an elliptic curve $E$ with Weil polynomial $f$, and, by Proposition~\ref{case11}, there exist cubic surfaces of types $(c_{11})$ and $(c_{12})$.
\end{example}

\subsection{Case $(c_{14})$}
We are going to construct an elliptic curve and a divisor $D$ such that there exist three inflection points $P_1,P_2$, and $P_3$ with the following properties:
\begin{enumerate}
	\item $P_1+_OP_2+_OP_3\neq Q$, i.e., the points do not lie on a line;
		\item the set $\{P_1,P_2,P_3\}$ is an orbit under the action of the Galois group $\Gal(\F_{q^3}/\F_q)$.
\end{enumerate}
Then, by Lemma \ref{inf9}, there exists a minimal cyclic cubic surface of type $(c_{14})$.

Assume that $q\equiv 1\bmod 3$. Take $b\in\Z$ such that $1-b+q\equiv 6\bmod 9$. 
Since $q>3$, there exists $b$ with an additional property that $(b,q)=1$, and $|b|\leqslant 2\sqrt{q}$. By Theorem~\ref{Deuring} there exists an elliptic curve $E$ 
with the Weil polynomial $f(t)=t^2-bt +q$. It follows that the $3$-torsion subgroup of $E(\F_q)$ is isomorphic to $\Z/3\Z$.
Moreover, $f(t)\equiv (t-1)^2\bmod 3$, thus Frobenius acts by a matrix 
$$\left (\begin{array}{c c}
1 & 0\\
1 & 1\\	
\end{array}\right)$$ in some basis.

From the relation $b\equiv q-5\bmod 9$ it follows that $f_3(t)=1-(b^3-3qb)+q^3$ is divisible by $27$. In other words the order of the group $E[9](\F_{q^3})$ is at least $27$, and there exists
a point $P\in E[9](\F_{q^3})$ of order $9$. 

We claim that $F(P)=P+Q'$ for some $Q'\in E[3](\F_{q^3})$. Indeed, if $F(P)=2P+Q'$, then $F^3(P)=2P+Q''$ for some $Q''\in E[3](\F_{q^3})$. We obtain a contradiction to the assumption that
 $P$ is defined over $\F_{q^3}$.
We may assume that $F(Q')=Q'+Q$, where $Q\in E[3](\F_q)$ is non-zero. If it is not the case, then pick a point $Q''\in E[3](\F_{q^3})$ such that $F(Q'')\neq Q''$, and take a point $P+Q''$ instead of $P$.
Finally, we claim that $3P=Q$. This can be easily deduced from the equation $F^2(P)-bF(P)+qP=0$. 

The divisor $D=2\cdot O+Q$ provides an embedding of $E$ into a projective plane. 
The orbit of $P$ under the Frobenius action is $\{P, P+_OQ', P+_O2Q'+_OQ\}$, and, since
$$
P+_OP+_OQ'+_OP+_O2Q'+_OQ=3P+_OQ=2Q\neq Q,
$$
\noindent these three points do not lie on a line. 

\section{Case $(c_{10})$}

In this section for $q > 2$ we construct cubic surfaces of type $(c_{10})$ and show that for $q = 2$ there are no cubic surfaces of type $(c_{10})$.

\begin{proposition}
\label{c10fq}
Let $\F_q$ be a finite field such that $q > 2$. Then there exists a cubic surface $X$ of type $(c_{10})$.
\end{proposition}

\begin{proof}

By Proposition \ref{twist-to-use}(2) it is sufficient to construct a cubic surface $X$ with an Eckardt point defined over $\F_q$ such that the generator $\gamma$ of the group $\Gamma$ has type $III$. This cubic surface is not minimal and it is isomorphic to the blowup of two triples of \mbox{$\Gal\left(\F_{q^3} / \F_q\right)$-conjugate} points $p_1$, $p_2$, $p_3$ and $p_4$, $p_5$, $p_6$ in general position on $\Pro^2_{\F_q}$. Moreover, three conjugate lines passing through $p_1$ and $p_4$, $p_2$ and $p_5$, $p_3$ and $p_6$ respectively must have a common $\F_q$-point,
 that is the image of the Eckardt point.

Assume that $q > 3$. Choose $k \in \F_q \setminus \{0,1,-1\}$. We claim that there exists $a \in \F_{q^3} \setminus \F_q$ such that six points
$$
p_1 = \left(a^2 : a : 1\right), \qquad p_2 = \left(a^{2q} : a^q : 1\right), \qquad p_3 = \left(a^{2q^2} : a^{q^2} : 1\right),
$$
$$
p_4 = \left(ka^2 : ka : 1\right), \qquad p_5 = \left(ka^{2q} : ka^q : 1\right), \qquad p_6 = \left(ka^{2q^2} : ka^{q^2} : 1\right)
$$
 are in general position.
 
If a conic passes through these six points, then it is defined over $\F_q$.
Assume that the points $p_1$ and $p_4$ lie on a conic  
$$
Ax^2 + Bxy + Cy^2 + Dxz + Eyz + Fz^2 = 0
$$ defined over $\F_q$.

 Then we have
$$
Aa^4 + Ba^3 + Ca^2 + Da^2 + Ea + F = 0, \quad Ak^2a^4 + Bk^2a^3 + Ck^2a^2 + Dka^2 + Eka + F = 0;
$$
$$
D(k^2-k)a^2 + E(k^2-k)a + F(k^2-1) = 0. 
$$
But the last equation holds if and only if $D(k^2-k) = E(k^2 - k) = F(k^2 - 1) = 0$. It means that $D = E = F = 0$ since $k \ne \pm 1$. Therefore $Aa^4 + Ba^3 + Ca^2 = 0$. But it holds only if $A= B = C = 0$. Thus the points $p_1$, $p_2$, $p_3$, $p_4$, $p_5$, $p_6$ do not lie on a conic.

Three lines passing through $p_1$ and $p_4$, $p_2$ and $p_5$, $p_3$ and $p_6$ respectively have a common point $(0:0:1)$.  There are~$q$ equations $x^{2q+1} - x^{q+2} = s$, where $s \in \F_q$. These equations have no more than $2q^2 + q$ roots in $\overline{\F}_q$. But $q^3 > 2q^2 + q$. Take $a\in \F_{q^3}$ such that $a^{2q+1} - a^{q+2}$ is not an element of $\F_q$. 
Assume that a line $l$ passes through $p_1$, $p_2$ and any other point $p_i$. If $l$ passes through $p_3$ then $l$ is defined over $\F_q$. But there are no lines defined over $\F_q$ passing through~$p_1$. If $l$ passes through $p_4$ or $p_5$, then $l$ is given by $x = ay$ or $x = a^qy$ respectively. But the points $p_2$ and $p_1$ respectively do not lie on the corresponding lines. The only remaining case is that $l$ passes through $p_1$, $p_2$ and $p_6$. In this case we have
$$
0 = \operatorname{det} \left(
\begin{array}{ccc} a^2&a&1\\a^{2q}&a^q&1\\ka^{2q^2}&ka^{q^2}&1\\
\end{array}
\right)
=
k \cdot\operatorname{det} \left(
\begin{array}{ccc} a^2&a&1\\a^{2q}&a^q&1\\a^{2q^2}&a^{q^2}&1\\
\end{array}
\right)
+
k \cdot\operatorname{det} \left(
\begin{array}{ccc} a^2&a&1\\a^{2q}&a^q&1\\0&0&k^{-1}-1\\
\end{array}
\right)
=
$$
$$
=
kD + (1 - k)(a^{q+2} - a^{2q+1}),
$$
\noindent where $k$, $D$ are elements of $\F_q$. Therefore $a^{2q+1} - a^{q+2}$ must be an element of $\F_q$, that contradicts the assumption $\left(a^{2q+1} - a^{q+2}\right) \notin \F_q$. Thus the line passing through $p_1$ and $p_2$ can not pass through any other point $p_i$. The points $p_1$, $p_2$ and $p_3$ are conjugate, therefore a line passing through any two points from this set does not contain any other point $p_i$. In the same way one can show, that any line passing through two points from the set $\{p_4, p_5, p_6\}$ does not contain any other point $p_i$. Thus the points $p_1$, $p_2$, $p_3$, $p_4$, $p_5$, $p_6$ are in general position.

Now assume that $q = 3$ and $a$ is an element of $\F_{q^3}$ such that $a^3 = a + 1$. Let us consider six points
$$
p_1 = \left( a^2 : a : 1 \right), \qquad p_2 = \left( a^6 : a^3 : 1 \right), \qquad p_3 = \left( a^{18} : a^9 : 1 \right),
$$
$$
p_4 = \left( a^4 : a : 1 \right), \qquad p_5 = \left( a^{12} : a^3 : 1 \right), \qquad p_6 = \left( a^{10} : a^9 : 1 \right).
$$

Three lines passing through $p_1$ and $p_4$, $p_2$ and $p_5$, $p_3$ and $p_6$ respectively have a common point $(1:0:0)$. Let us show that the points $p_1$, $p_2$, $p_3$, $p_4$, $p_5$, $p_6$ are in general position.

If a conic pass through these six points then it is defined over $\F_3$. Let us consider a conic given by the equation
$$
Ax^2 + Bxy + Cy^2 + Dxz + Eyz + Fz^2 = 0.
$$
Assume that the points $p_1$ and $p_4$ lie on this conic. Then
$$
Aa^4 + Ba^3 + Ca^2 + Da^2 + Ea + F = 0, \qquad Aa^8 + Ba^5 + Ca^2 + Da^4 + Ea + F = 0.
$$
Since $a^3 = a + 1$ we have
$$
(A + C + D)a^2 + (A + B + E)a + (B + F) = 0, \quad (B + C + D - A)a^2 + (B + D + E)a + (B + F - A) = 0.
$$
One can check that these equations holds only if $A = B = C = D = E = F = 0$. Thus the points $p_1$, $p_2$, $p_3$, $p_4$, $p_5$, $p_6$ do not lie on a conic.

If a line passes through the three points $p_1$, $p_2$, $p_3$ or the points $p_4$, $p_5$, $p_6$, then this line is defined over $\F_3$. Let us consider a line $l$ defined over $\F_3$ given by the equation $Ax + By + Cz = 0$. If this line passes through $p_1$ then $Aa^2 + Ba + C = 0$ that is impossible. If $l$ passes through $p_4$ then $Aa^4 + Ba + C = 0$. Thus $Aa^2 + (A+B)a + C = 0$ that is impossible.

Note that each of the lines $y = az$, $y = a^3z$ and $y = a^9z$ contains exactly two points from the set $p_1$, $p_2$, $p_3$, $p_4$, $p_5$, $p_6$. 
Thus it is sufficient to consider triples $p_1$, $p_2$, $p_6$ and $p_4$, $p_5$, $p_3$.

For the points $p_1$, $p_2$, $p_6$ we have
$$
\operatorname{det} \left(
\begin{array}{ccc} a^2&a&1\\a^{6}&a^3&1\\a^{10}&a^9&1\\
\end{array}
\right)
= a^{15} + a^{11} + a^5 - a^{13} - a^{11} - a^7 = a^5(a^2 - 1)(a^8 - 1) = 
$$
$$
= a^5(a^2 - 1)^2(a^2 + 1)(a^4 + 1) = a^5(a^2 - 1)^2(a^2 + 1)(a^2 + a + 1) \ne 0.
$$

For the points $p_4$, $p_5$, $p_3$ we have
$$
\operatorname{det} \left(
\begin{array}{ccc} a^4&a&1\\a^{12}&a^3&1\\a^{18}&a^9&1\\
\end{array}
\right)
= a^{21} + a^{19} + a^{7} - a^{21} - a^{13} - a^{13} = a^{19} - 2a^{13} + a^{7} = 
$$
$$
= a^7(a^6 - 1)^2 = a^7(a^3 - 1)^2(a^3 + 1)^2 = a^9(a - 1)^2 \ne 0.
$$

Thus the points $p_1$, $p_2$, $p_3$, $p_4$, $p_5$, $p_6$ are in general position.

\end{proof}

From now assume that $q = 2$.

\begin{lemma}
\label{c10f2eck}
If a cubic surface $X$ of type $(c_{10})$ is defined over $\F_2$, then it does not contain Eckardt points defined over $\F_2$.
\end{lemma}

\begin{proof}
By Proposition \ref{twist-to-use}(2), there exists a cubic surface $X'$ with an Eckardt point defined over $\F_2$ such that the generator $\gamma$ of the group $\Gamma'$ has type $III$. This cubic surface is not minimal and it is isomorphic to the blowup of $\Pro^2_{\F_2}$ at two triples of \mbox{$\Gal\left(\F_{8} / \F_2\right)$-conjugate} points $p_1$, $p_2$, $p_3$ and $p_4$, $p_5$, $p_6$ in general position. Moreover, three conjugate lines passing through $p_1$ and $p_4$, $p_2$ and $p_5$, $p_3$ and~$p_6$ respectively must have a common $\F_2$-point $P$ which is the image of the Eckardt point.

Let us consider any three conjugate lines $l_1$, $l_2$ and $l_3$ defined over $\F_8$ and passing through $\F_2$-point $P$. Assume that three conjugate points $p_1 \in l_1$, $p_2 \in l_2$ and $p_3 \in l_3$ do not lie on a line. There are $7$ lines on $\Pro^2_{\F_2}$ defined over $\F_2$ and $7$ conics on $\Pro^2_{\F_2}$ passing through $p_1$, $p_2$ and $p_3$ defined over $\F_2$. Three of these lines and three of these conics pass through $P$. Therefore there are four lines defined over $\F_2$ and not passing through $P$, and four conics defined over $\F_2$, not passing through $P$ and passing through $p_1$, $p_2$ and~$p_3$. Note that each two of these eight curves do not have a common point lying on one of the lines $l_1$, $l_2$ or $l_3$ since otherwise they have a common point on each of these three lines. Therefore eight of $\F_8$-points on $l_1$ lie on these eight curves and the ninth point is~$P$. Thus for any three conjugate points $p_4 \in l_1$, $p_5 \in l_2$ and $p_6 \in l_3$, the points $p_1$, $p_2$, $p_3$, $p_4$, $p_5$, $p_6$ are not in general position.

\end{proof}

\begin{lemma}
\label{onepointcubic}
If a plane cubic curve $S$ over $\F_2$ contains exactly one $\F_2$-point that is singular, then $S$ is a union of three lines defined over $\F_8$.
\end{lemma}

\begin{proof}
Assume that a plain cubic curve $S$ has unique $\F_2$-point $(0:0:1)$ that is singular. Then $S$ is given by the equation
$$
\left(Ax^3 + Bx^2y + Cxy^2 + Dy^3\right) + \left(Ex^2 + Fxy + Gy^2\right)z = 0.
$$
The points $(1:0:0)$, $(0:1:0)$ and $(1:1:0)$ do not lie on $S$, therefore $A = D = 1$ and $B + C = 1$. Without loss of generality we can assume that $S$ is given by the equation
$$
x^3 + xy^2 + y^3 + \left(Ex^2 + Fxy + Gy^2\right)z = 0.
$$
The points $(1:0:1)$, $(0:1:1)$ and $(1:1:1)$ do not lie on $S$, therefore $E = G = F = 0$ and $S$ is given by the equation
$$
x^3 + xy^2 + y^3 = 0,
$$
\noindent that gives a triple of lines passing through $(0:0:1)$ and defined over $\F_8$.

\end{proof}

\begin{corollary}
\label{onepointtangent}
If a cubic surface $X$ of type $(c_{10})$ is defined over $\F_2$ then any tangent plane at $\F_2$-point contains at least two $\F_2$-points.
\end{corollary}

\begin{proof}
Assume that there is an $\F_2$-point $P$ such that the tangent plane $T_P(X)$ contains one $\F_2$-point. Then the section of $X$ by the plane $T_P(X)$ is a plane cubic curve that has a unique $\F_2$-point $P$ which is singular. By Lemma \ref{onepointcubic} the point $P$ is an Eckardt point, that contradicts Lemma \ref{c10f2eck}.
\end{proof}

\begin{proposition}
\label{c10f2}
There are no cubic surfaces of type $(c_{10})$ defined over $\F_2$.
\end{proposition}

\begin{proof}
Let $X$ be a cubic surface over $\F_2$ of type $(c_{10})$. Then
$$
Z_X(t)=\frac{1}{(1-t)(1-2t)(1+2t)^2(1+2t+4t^2)(1-2t+4t^2)(1-4t)}.
$$
Therefore there are exactly three $\F_2$-points on $X$.

Let $X$ contain three $\F_2$-points lying on a line. Then by Corollary \ref{onepointtangent} the tangent plane at any of these points passes through the two other $\F_2$-points. Therefore the line passing through the three $\F_2$-points defined over $\F_2$ lies on $X$ that contradicts the minimality of~$X$.

Let $X$ contain three $\F_2$-points not lying on a line. We can assume that these points are $(1:0:0:0)$, $(0:1:0:0)$ and $(0:0:1:0)$. Then in the plane $t = 0$ each tangent line of a nonsingular $\F_2$-point must pass through other $\F_2$-point by Corollary \ref{onepointtangent} and Lemma~\ref{c10f2eck}. Therefore without loss of generality the section $t = 0$ is given by $x^2y + y^2z + z^2x = 0$.

Let us consider a cubic surface given by the equation
$$
x^2y + y^2z + z^2x + Q(x, y, z)t + L(x, y, z)t^2 + Dt^3 = 0,
$$
\noindent where $Q(x, y, z)$ is a quadratic polynomial, $L(x, y, z)$ is a linear polynomial and $D$ is a constant. The surface $X$ does not contain any $\F_2$-points not lying on the plane section $t = 0$. Therefore one can check that $X$ is given by the equation
$$
x^2y + y^2z + z^2x + (xy + xz + yz)t + A(x^2t + xt^2) + B(y^2t + yt^2) + C(z^2t + zt^2) + t^3 = 0,
$$
\noindent where $A$, $B$ and $C$ are coefficients. By replacing the coordinates
$$
x \mapsto x + t, \qquad y \mapsto y + t, \qquad z \mapsto z + t
$$
\noindent we can make $A$, $B$ and $C$ equal to zero. Therefore $X$ can be given by the equation
$$
x^2y + y^2z + z^2x + (xy + xz + yz)t + t^3 = 0.
$$
\noindent Note that the points $(\xi_7 : \xi_7^2 : \xi_7^4 : 1)$, $(\xi_7^2 : \xi_7^4 : \xi_7 : 1)$ and $(\xi_7^4 : \xi_7 : \xi_7^2 : 1)$ defined over $\F_8$, where $\xi_7^3 + \xi_7 + 1 = 0$, are singular points on $X$. But $X$ must be a smooth cubic surface. This contradiction finishes the proof. 

\end{proof}

\appendix

\section{Conjugacy classes of elements in $W(E_6)$}

In Table \ref{table1} we collect some facts about conjugacy classes of elements in the Weyl group~$W(E_6)$. This table based on tables in \cite{SD}, \cite{Man74} and \cite{Car72}. The first column is a type according to \cite{SD}. The second column is the Carter graph corresponding to conjugacy class (see \cite{Car72}). The third column is the order of element. The fourth column is the collection of eigenvalues of the action of element on $K_X^{\perp} \subset \Pic(\XX) \otimes \mathbb{Q}$. In Remarks \ref{S6_in_W6} and \ref{3notation} we introduce notation for representatives of some conjugacy classes. For convenience of the reader we give this notation in the last column of the table.

\begin{table}
\caption{} \label{table1}

\begin{center}
\begin{tabular}{|c|c|c|c|c|}
\hline
Type & Graph & Order & Eigenvalues & Names \\
\hline
$c_1$ & $\varnothing$ & $1$ & $1$, $1$, $1$, $1$, $1$, $1$ & $\mathrm{id}$ \\
\hline
$c_2$ & $A_1^2$ & $2$ & $-1$, $-1$, $1$, $1$, $1$, $1$ & $(12)(34)$ \\
\hline
$c_3$ & $A_1^4$ & $2$ & $-1$, $-1$, $-1$, $-1$, $1$, $1$ & \\
\hline
$c_4$ & $D_4(a_1)$ & $4$ & $i$, $i$, $-i$, $-i$, $1$, $1$ & \\
\hline
$c_5$ & $A_3 \times A_1$ & $4$ & $i$, $-i$, $-1$, $-1$, $1$, $1$ & $(1234)(56)$ \\
\hline
$c_6$ & $A_2$ & $3$ & $\omega$, $\omega^2$, $1$, $1$, $1$, $1$ & $(123)$, type $II$ \\
\hline
$c_7$ & $D_4$ & $6$ & $-\omega$, $-\omega^2$, $-1$, $-1$, $1$, $1$ & \\
\hline
$c_8$ & $A_2 \times A_1^2$ & $6$ & $\omega$, $\omega^2$, $-1$, $-1$, $1$, $1$ & \\
\hline
$c_9$ & $A_2^2$ & $3$ & $\omega$, $\omega$, $\omega^2$, $\omega^2$, $1$, $1$ & $(123)(456)$, type $III$ \\
\hline
$c_{10}$ & $A_5 \times A_1$ & $6$ & $-\omega$, $-\omega^2$, $\omega$, $\omega^2$, $-1$, $-1$ & \\
\hline
$c_{11}$ & $A_2^3$ & $3$ & $\omega$, $\omega$, $\omega$, $\omega^2$, $\omega^2$, $\omega^2$ & type $I$ \\
\hline
$c_{12}$ & $E_6(a_2)$ & $6$ & $-\omega$, $-\omega$, $-\omega^2$, $-\omega^2$, $\omega$, $\omega^2$ & \\
\hline
$c_{13}$ & $E_6$ & $12$ & $i\omega$, $i\omega^2$, $-i\omega$, $-i\omega^2$, $\omega$, $\omega^2$ & \\
\hline
$c_{14}$ & $E_6(a_1)$ & $9$ & $\xi_9$, $\xi_9^2$, $\xi_9^4$, $\xi_9^5$, $\xi_9^7$, $\xi_9^8$ & \\
\hline
$c_{15}$ & $A_4$ & $5$ & $\xi_5$, $\xi_5^2$, $\xi_5^3$, $\xi_5^4$, $1$, $1$ & $(12345)$ \\
\hline
$c_{16}$ & $A_1$ & $2$ & $-1$, $1$, $1$, $1$, $1$, $1$ & $(12)$ \\
\hline
$c_{17}$ & $A_1^3$ & $2$ & $-1$, $-1$, $-1$, $1$, $1$, $1$ & $(12)(34)(56)$ \\
\hline
$c_{18}$ & $A_3$ & $4$ & $i$, $-i$, $-1$, $1$, $1$, $1$ & $(1234)$ \\
\hline
$c_{19}$ & $A_3 \times A_1^2$ & $4$ & $i$, $-i$, $-1$, $-1$, $-1$, $1$ & \\
\hline
$c_{20}$ & $D_5$ & $8$ & $\xi_8$, $\xi_8^3$, $\xi_8^5$, $\xi_8^7$, $-1$, $1$ & \\
\hline
$c_{21}$ & $A_2 \times A_1$ & $6$ & $\omega$, $\omega^2$, $-1$, $1$, $1$, $1$ & $(123)(45)$ \\
\hline
$c_{22}$ & $A_2^2 \times A_1$ & $6$ & $\omega$, $\omega$, $\omega^2$, $\omega^2$, $-1$, $1$ & \\
\hline
$c_{23}$ & $A_5$ & $6$ & $-\omega$, $-\omega^2$, $\omega$, $\omega^2$, $-1$, $1$ & $(123456)$ \\
\hline
$c_{24}$ & $D_5(a_1)$ & $12$ & $-\omega$, $-\omega^2$, $i$, $-i$, $-1$, $1$ & \\
\hline
$c_{25}$ & $A_4 \times A_1$ & $10$ & $\xi_5$, $\xi_5^2$, $\xi_5^3$, $\xi_5^4$, $-1$, $1$ & \\
\hline

\end{tabular}
\end{center}

\end{table}

\bibliographystyle{alpha}

\end{document}